\documentclass[12pt]{amsart}
\usepackage{amssymb,verbatim,enumerate,ifthen}
\usepackage{esint}
\usepackage[mathscr]{eucal}
\usepackage[utf8]{inputenc}
\usepackage[T1]{fontenc}
\makeatletter
\@namedef{subjclassname@2010}{%
  \textup{2020} Mathematics Subject Classification}
\makeatother

\newtheorem{thm}{Theorem}
\newtheorem{cor}{Corollary}
\newtheorem{prop}{Proposition}
\newtheorem{lem}{Lemma}

\newcommand{\Hc}{\mathscr{H}}
\newcommand{\Est}{\mathscr{C}}

\newtheorem{rem}{Remark}
\newtheorem{exa}{Example}

\numberwithin{equation}{section}

\newcommand\nolabel[1]{\nonumber}

\newcommand\A{\mathscr{A}}
\newcommand\R{\mathbb{R}}
\renewcommand\P{\mathscr{P}}
\newcommand\N{\mathbb{N}}

\newcommand{\M}{\mathscr{M}}
\newcommand{\Nm}{\mathscr{N}}

\newcommand{\abs}[1]{\left| #1 \right| }
\newcommand{\ceil}[1]{\lceil #1 \rceil}

\DeclareFontFamily{U}{mathb}{}
\DeclareFontShape{U}{mathb}{m}{n}{
  <-5.5> mathb5
  <5.5-6.5> mathb6
  <6.5-7.5> mathb7
  <7.5-8.5> mathb8
  <8.5-9.5> mathb9
  <9.5-11.5> mathb10
  <11.5-> mathb12
}{}
\DeclareSymbolFont{mathb}{U}{mathb}{m}{n}
\DeclareMathSymbol{\csqs}{\mathbin}{mathb}{"0D}% name to be checked

\author[P. Pasteczka]{Pawe\l{} Pasteczka}
\address{Institute of Mathematics \\ Pedagogical University of Krak\'ow \\ Podchor\k{a}\.zych str. 2, 30-084 Krak\'ow, Poland}
\email{pawel.pasteczka@up.krakow.pl}

    

\newcommand{\operator}[1]{\mathop{\vphantom{\sum}\mathchoice
{\vcenter{\hbox{\LARGE $#1$}}}
{\vcenter{\hbox{\Large $#1$}}}{#1}{#1}}\displaylimits}

\def\Mst_#1^#2{\operator{\mathscr{M}_{\mbox{\scriptsize$\#$}}\!\!}_{#1}^{#2}\,\,}

\def\Mm{\operator{\mathscr{M}}}
       
\numberwithin{equation}{section}
\def\eq#1{{\rm(\ref{#1})}}
\def\Eq#1#2{\ifthenelse{\equal{#1}{*}}
  {\begin{equation*}\begin{aligned}[]#2\end{aligned}\end{equation*}}
  {\begin{equation}\begin{aligned}[]\label{#1}#2\end{aligned}\end{equation}}}

\title{On the Hardy property of mixed means}
\begin{document}

\begin{abstract}
 Hardy property of means has been extensively studied by P\'ales and Pasteczka since 2016. The core of this research is based on few of their properties: concavity, symmetry, monotonicity, repetition invariance and homogeneity (last axiom was recently omitted using some homogenizations techniques). In the present paper we deliver a study of possible omitting monotonicity and replacing repetition invariance by a weaker axiom. 
 
 These results are then used to establish the Hardy constant for certain types of mixed means.
\end{abstract}

\maketitle

\section{Introduction}
The notion of Hardy means was formally introduced by P\'ales-Persson in 2004 \cite{PalPer04}, however its origin goes back to 1920s/30s when there appear a series of papers by, among others, Hardy \cite{Har20}, Landau \cite{Lan21}, Knopp \cite{Kno28}, and Carleman \cite{Car32}.  In order to present their results in an appropriate setup recall that for $p\in\R$ the $p$th power mean of the positive numbers $x_1,\dots,x_n$ equals
\Eq{PM}{
  \P_p(x_1,\dots,x_n)
   :=\left\{\begin{array}{ll}
    \Big(\dfrac{x_1^p+\cdots+x_n^p}{n}\Big)^{\frac{1}{p}} 
      &\mbox{if }p\neq0, \\[3mm]
      \sqrt[n]{x_1\cdots x_n}\qquad
      &\mbox{if }p=0.
    \end{array}\right.
}

Then all early results mentioned above can be expressed in a compact form
\Eq{*}{
\sum_{n=1}^\infty \P_p(x_1,\dots,x_n) \le \gamma_p \sum_{n=1}^\infty x_n \qquad \text{ for all }(x_n)_{n=1}^\infty \in \ell_1(\R_+),
}
where $p \in (-\infty,1)$ and
\Eq{*}{
\gamma_p:=
\begin{cases}
      \big(1-p\big)^{-\frac1p} &\mbox{if }p \in (-\infty,1)\setminus\{0\}, \\[1mm]
      e &\mbox{if }p=0.
\end{cases}
}
Moreover all these constants are sharp, i.e. they cannot be diminished. 

As~the matter of fact this area has been active ever since that period. Detailed history is presented in papers Pe\v{c}ari\'c--Stolarsky \cite{PecSto01}, Duncan--McGregor \cite{DunMcg03}, and the book of Kufner--Maligranda--Persson \cite{KufMalPer07fixed}. 
There are a number of generalizations and related results, however we are especially interest in two of them. First, following P\'ales and Persson \cite{PalPer04}, mean $\M \colon \bigcup_{n=1}^{\infty} I^n \to I$ (here and in the sequel $I\subseteq\R$ is a nondegenerated interval with $\inf I=0$) is called a \emph{Hardy mean} (or has a Hardy property) if there exists a constant $C \in [1,+\infty)$ such that
\Eq{E:Hc}{
\sum_{n=1}^\infty \M(x_1,\dots,x_n) \le C \sum_{n=1}^\infty x_n \qquad \text{ for all }(x_n)_{n=1}^\infty \in \ell_1(I).
}
To clarify this definition recall that a mean (on $I$) is an arbitrary function $\M \colon \bigcup_{n=1}^{\infty} I^n \to I$ with $\min \le \M \le \max$.

Now, due to P\'ales and Pasteczka \cite{PalPas16} we define the \emph{Hardy constant of $\M$} as the smallest extended real number $C$ satisfying property \eq{E:Hc} and denote it as $\Hc(\M)$. In this manner a mean admit the Hardy property if and only if it has a finite Hardy constant. Furthermore $\Hc(\P_p)=\gamma_p$ for all $p \in (-\infty,1)$.

In a finite form, following \cite{PalPas16}, for $n\in\N$, we define $\Hc_n(\M)$ to be the smallest nonnegative number such that
\Eq{E:HnM}{
\M(x_1)+\dots+\M(x_1,\dots,x_n) \le \Hc_n(\M)(x_1+\cdots+x_n)
}
holds for all $(x_1,\dots,x_n) \in I^n$. The sequence $\big(\Hc_n(\M)\big)_{n=1}^\infty$ will be called the \emph{Hardy sequence of $\M$}. Due to the mean value property, we easily obtain the inequality $1\leq\Hc_n(\M)\leq n$.  
This sequence was of some interest among the years. For example Kaluza and Szeg\H{o} 
\cite{KalSze27} proved $\Hc_n(\P_p) \le \tfrac{1}{n(\exp(1/n)-1)} \cdot \gamma_p$ for $p \in [0,1)$ and $n \in 
\N$. Furthermore it is known \cite[p.~267]{HarLitPol34} that $\Hc_n(\P_0) \le (1+\tfrac1n)^n$ for all $n \in \N$. There are also a number of other results like \cite{WuZhaWan08} where the approximate values of $(\Hc_n(\P_0))_{n=1}^{12}$ was given. We are not going to recall them in details as they are outside the scope of this paper.

Hardy sequences have an interesting limit behavior.  Let us put this as a proposition since we are going to refer to it further.
\begin{prop}[\cite{PalPas16}, Proposition 3.1]
\label{prop:Hn->Hinfty}
For every mean $\M \colon \bigcup_{n=1}^{\infty} I^n \to I$, its Hardy sequence is nondecreasing and converges to $\Hc(\M)$.
\end{prop}

There appear a natural question how to obtain the Hardy constant for a given mean. The first important step was to find the lower bound of $\Hc(\M)$. The idea was to use the Stolz–Ces\`aro theorem to the series $\sum x_n$ and $\sum \M(x_1,\dots,x_n)$. This lead us to the next proposition.
\begin{prop}[\cite{PalPas16}, Theorem~3.3]
\label{prop:lower_estim_Hc}
Let $\M \colon \bigcup_{n=1}^{\infty} I^n \to I$ be a mean. Then, for all sequences $(x_n)_{n=1}^\infty$ in $I$ that 
does not belong to $\ell_1$,
\Eq{*}{
  \liminf_{n \to \infty} x_n^{-1}\M(x_1,\ldots,x_n)\leq\Hc(\M).
}
\end{prop}
If we apply this proposition to the family of harmonic sequences $\{(\tfrac yn)_{n=1}^\infty \colon y \in I\}$ we obtain the 
following lower estimation of $\Hc(\M)$.
\begin{prop}\label{prop:LowHar}
Let $\M \colon \bigcup_{n=1}^{\infty} I^n \to I$ be a mean. Then $\Hc(\M) \ge \Est(\M)$, where
\Eq{E:EstM}{
\Est(\M):=\sup_{y\in I} \liminf_{n \to \infty} \frac ny \cdot \M\left(\frac y1,\frac y2,\ldots,\frac yn 
\right).
}
\end{prop}

The problem of calculating $\Est(\M)$ seams to be significantly easier than the Hardy constant. Nevertheless some results contained in \cite{PalPas16} shown that these two constants are closely related to each other in a broad class of means (see Proposition~\ref{prop:PalPas16Thm3.4} below). 

\medskip
Let us emphasize that a number of results in \cite{PalPas16} which refers to Hardy property claim so-called repetition invariance of mean. We aim to relax this assumption to repetition superinvariace (which is a new definition). Later we show some examples of repetition superinvariant means which are not repetition invariant and establish their Hardy constant.

Just for technical reason, using some recent results concerning homogenizations \cite{PalPas19a}, we restrict our consideration to $I=\R_+$ (without any loss of generality) which provide us some additional properties. Later, we reproved few results from \cite{PalPas16} in this more general (superinvariant) setting and, finally, establish Hardy properties among two broad families of mixed power means. 

\section{Properties of means}

Let us now introduce some important properties of means. 
We begin with few conventions which help us to avoid misunderstandings. Let $\N:=\{1,2,\dots\}$ and for $n \in \N$ define $\N_n:=\{1,\dots,n\}$, as it is handy. 

A mean $\M$ on $I$ is said to be \emph{symmetric}, if for all $n \in \N$, 
$x \in I^n$, and a permutatuion $\sigma \colon \N_n \to \N_n$, the equality $\M(x) =\M(x\circ\sigma)$ is valid.
We~call a mean $\M$ to be \emph{Jensen concave} if, for all $n \in \N$, its restriction $\M|_{I^n}$ is Jensen concave.
We can introduce Jensen convexity in a similar manner. Note that, since $\M$ is locally bounded, the Bernstein--Doetsch Theorem \cite{BerDoe15} implies that its Jensen concavity (convexity) is in fact equivalent to concavity (convexity). 
A mean $\M$ is said to be \emph{monotone} (or \emph{nondecreasing}) if for all $n \in \N$, the restriction $\M|_{I^n}$ is nondecreasing in each of its entry. Assuming that $I = \R_+$, we call a mean $\M$ \emph{homogeneous}, if for all $t>0$, $n\in\N$ and 
$x \in \R_+^n$, we have $\M(tx)=t\M(x)$. Mean $\M$ is called \emph{repetition invariant} if, for all $n,m\in\N$
and $(x_1,\dots,x_n)\in I^n$, the following identity is satisfied
\Eq{*}{
  \M(\underbrace{x_1,\dots,x_1}_{m\text{-times}},\dots,\underbrace{x_n,\dots,x_n}_{m\text{-times}})
   =\M(x_1,\dots,x_n).
}

Finally a symmetric mean $\M$ is \emph{associative} if for all $n,m\in\N$, and a pair of vectors $(x_1,\dots,x_n) \in I^n$ and $y=(y_1,\dots,y_m) \in I^m$ we have
\Eq{*}{
\M(x_1,\dots,x_n,y_1,\dots,y_m)=\M(x_1,\dots,x_n,\underbrace{\mu,\dots,\mu}_{m\text{-times}})\text{, where }\mu:=\M(y).
}
This axiom is very characteristic for quasiarithmetic means (see \cite{Kol30}). Furthermore it is easy to show that every associative mean is repetition invariant.

To conclude this section let us go back to the Hardy property. As it was already mentioned, due to \cite{PalPas16} the Hardy constant is linked to $\Est(\M)$ defined by \eq{E:EstM}. To clarify this link we recall the important result binding them even stronger than Proposition~\ref{prop:LowHar}. 

\begin{prop}[\cite{PalPas16}, Theorem 3.4]\label{prop:PalPas16Thm3.4}
Let $\M \colon \bigcup_{n=1}^{\infty} \R_+^n \to\R_+$ be an increasing, symmetric, repetition invariant, and Jensen 
concave mean. Then $\Hc(\M)=\Est(\M)$.
\end{prop}

\subsection{Homogenizations} 

%\cite{PalPas19b}%,PalPas20}
We also use the notion of homogenization from \cite{PalPas19a}. The idea was to associate the homogeneous mean to a given one. 
Indeed, according to \cite{PalPas19a}, for a mean $\M$ on an interval $I$ (recall that we claim $\inf I=0$) we introduce the (local) homogenizations
$\M_\#, \M^\# \colon \bigcup_{n=1}^\infty \R_+^n \to \R_+$ by 
\Eq{*}{
\M_\#(x) := \liminf_{t \to 0} \tfrac1t \M(tx)
\quad \text{ and }\quad
\M^\#(x) := \limsup_{t \to 0} \tfrac1t \M(tx)\:.
}
In the case when $\M$ is Jensen concave these two means coincide with each other, which can be expressed formally in the following statement.
\begin{prop}[\cite{PalPas19a}, Theorem~2.1]\label{prop:PalPas19aThm2.1}
Let $\M$ be a Jensen concave mean on $I$. Then $\M_\# = \M^\#$ and these means are also Jensen concave. 

In addition, $\M \le \M_\# = \M^\#$ on the domain of $\M$.
\end{prop}

Let us now recall a property binding the Hardy constant of the mean and its homogenizations.
\begin{prop}[\cite{PalPas20}, Theorem~3.3] For every mean $\M$ on $I$ we have $\Hc(\M_\#)\le \Hc(\M)$. 

Moreover, if $\M$ is Jensen concave then $\Hc(\M_\#)=\Hc(\M^\#)= \Hc(\M)$.
\end{prop}
\subsection{Auxiliary results}
As we are going to generalize consideration enclosed in \cite{PalPas16}, we need to go inside the proof of the main theorem contained therein -- more precisely \cite[proof of Theorem~2.1]{PalPas16}. It was based on the result by Kedlaya \cite{Ked94}. The key tool was the existence of a family of matrices with some specific properties. Since we need to use them in the new setting let us isolate this combinatorial fact.
\begin{lem}\label{lem:matrix}
 For all $n \in \N$ there exists a matrix $K \in {\N_n}^{n! \times n!}$ such that:
 \begin{enumerate}[(i)]
  \item for all $s \in \N_n$ and $p \in \N_{n!}$ the number of appearance of the value $s$ in the $p$-th row of $K$ equals
\Eq{*}{
\alpha_p(s):=\begin{cases} 
  n!\ceil{\frac{p}{(n-1)!}}^{-1} &\quad \text{ if }s\le \ceil{\frac{p}{(n-1)!}},\\
  0 &\quad \text{ otherwise;}
  \end{cases}
}
  \item for all $s \in \N_n$ and $q \in \N_{n!}$ the number of appearance of the value $s$ in the $q$-th column of $K$ equals
$\alpha_q(s)$.
 \end{enumerate}
\end{lem}

Remarkably, this lemma does not involve means in its wording. It is very important for us since it has no superfluous assumptions. Lemma~\ref{lem:matrix} in a similar form was already used in \cite{ChuPalPas21}. It was also used in proving a sort of Kelaya's inequality for concave means in both \cite{Ked94} and \cite{PalPas16}.

Before we present next statement let us introduce a handy, sum-type notion. For a mean $\M$ on $I$, a vector of entries $x\in I^n$ and a nonzero vector of nonnegative integers $\lambda \in (\N \cup \{0\})^n  \setminus \{(0,\dots,0)\}$ we denote briefly
\Eq{*}{
\Mm_{i=1}^n (x_i,\lambda_i) := \M(\underbrace{x_1,\dots,x_1}_{\lambda_i\text{-times}},\dots,\underbrace{x_n,\dots,x_n}_{\lambda_n\text{-times}}).
}
Using this notion we formulate and proof a variation on \cite[Theorem~2.1]{PalPas16}. In the original result there was an additional assumption (repetition invariance of $\M$).

\begin{prop}\label{prop:Ked}
 Let $\M$ be a symmetric and Jensen-concave mean on~$I$. Then   for all $n \in \N$ and $x\in I^n$ we have
 \Eq{*}{
\sum_{i=1}^n \Mm_{j=1}^i \Big( x_j , \tfrac{n!}i\Big) \le n \Mm_{i=1}^n \Big(\frac{x_1+\dots+x_i}i, (n-1)!\Big)\,.
 }
\end{prop}
\begin{proof}
 Fix $n \in \N$ and $x \in I^n$ arbitrarily. Let $K\in {\N_n}^{n! \times n!}$ be the matrix from Lemma~\ref{lem:matrix}. Define the matrix $A=(a_{p,q}) \in I^{n!\times n!}$ by $a_{p,q}:=x_{K_{p,q}}$. Then since $\M$ is Jensen-concave we obtain
 \Eq{EKK1}{
\A\big(\M\big( a_{p,q} \colon q \in \N_{n!}\big) \colon p \in \N_{n!}\big) 
\le 
\M\big(\A\big( a_{p,q} \colon p \in \N_{n!}\big) \colon q \in \N_{n!}\big).
%=\Mm_{q=1}^{n!}\Ar_{p=1}^{n!}a_{p,q}. %\Mm_{q=1}^{n!}(  \frac{1}{n!} \sum_{p=1}^{n!} a_{p,q}).
 }
Now using the property (i) in Lemma~\ref{lem:matrix}, since $\M$ is symmetric we have
\Eq{*}{
%\Mm_{q=1}^{n!}a_{p,q}=\Mm_{q=1}^{n!}(x_{K_{p,q}})
\M\big( a_{p,q} \colon q \in \N_{n!}\big)
=\M\big( x_{K_{p,q}} \colon q \in \N_{n!}\big)
=\Mm_{j=1}^{n}(x_j,\alpha_p(j))
\quad \text{ for all }p \in \N_{n!}.
}
Thus, using the definition of $\alpha_p$ we can omit all terms with zero entries to obtain
\Eq{*}{
\M\big( a_{p,q} \colon q \in \N_{n!}\big)=\Mm_{j=1}^{\ceil{\frac{p}{(n-1)!}}}\bigg(x_j,\frac{n!}{\ceil{\tfrac{p}{(n-1)!}}}\bigg)
\quad \text{ for all }p \in \N_{n!}.
}
Now we can sum-up the above equality over $p \in \N_{n!}$. Then the right hand side naturally splits into $n$ blocks of cardinality $(n-1)!$ and, after dividing by $n!$ \mbox{side-by-side}, we arrive at 
\Eq{EKK2}{
\A\big(\M\big( a_{p,q} \colon q \in \N_{n!}\big) \colon p \in \N_{n!}\big) =\frac1n \sum_{i=1}^n\Mm_{j=1}^{i}\Big(x_j,\tfrac{n!}i\Big).
}
If we interchange $\M$ with $\A$, and reapply above consideration to the matrix $A^T$, we get
\Eq{EKK3}{
\M\big(\A\big( a_{p,q} \colon p \in \N_{n!}\big) \colon q \in \N_{n!}\big)
%\Mm_{q=1}^{n!}\Ar_{p=1}^{n!}a_{p,q}
=\Mm_{i=1}^n \Big(\frac{x_1+\dots+x_i}i, (n-1)!\Big).
}
Finally, binding \eq{EKK2}, \eq{EKK1}, and \eq{EKK3} we obtain
\Eq{*}{
\frac1n \sum_{i=1}^n\Mm_{j=1}^{i}\Big(x_j,\tfrac{n!}i\Big)
%\Ar_{i=1}^n\Mm_{j=1}^{i}\Big(x_j,\tfrac{n!}i\Big)
&= \A\big(\M\big( a_{p,q} \colon q \in \N_{n!}\big) \colon p \in \N_{n!}\big) \\
&\le \M\big(\A\big( a_{p,q} \colon p \in \N_{n!}\big) \colon q \in \N_{n!}\big) \\
&=\Mm_{i=1}^n \Big(\frac{x_1+\dots+x_i}i, (n-1)!\Big),
}
which is trivially equivalent to our assertion.
\end{proof}

Let us conclude this section with a general property of Jensen-concave means defined on a positive half-line.

\begin{lem}\label{lem:conc->mon}
% If $I$ is an interval with $\sup I=+\infty$ then 
Every Jensen-concave mean on an interval $I$ with $\sup I=+\infty$ is monotone.
\end{lem}
\begin{proof}
Let $\M\colon \bigcup_{n=1}^\infty I^n \to I$ be a Jensen-concave mean. Assume to the contrary that there exist $n \in \N$, and vectors $v \in I^n$, $w \in \R_+^n$ such that
$\delta:=\M(v)-\M(v+w)>0$. Then, since $\M$ is concave, we have
\Eq{*}{
\M(v+w) \ge \frac{n\M(v)+\M(v+(n+1)w)}{n+1} \qquad \text{ for all }n\in \N.
}
Thus, by mean property we obtain
\Eq{*}{
\min(v) \le \M(v+(n+1)w) &\le (n+1) \M(v+w) -n \M(v)\\
&= \M(v+w)-n\delta \qquad \text{ for all }n \in\N.
}
In the limit case as $n \to \infty$ we obtain $\min(v)=-\infty$ contradicting the choice of $v$. 
\end{proof}

\section{Main results}
We are now heading towards the main result of this paper. In order to provide the proper setting we introduce a weaker form of repetition invariance. Namely, the mean $\M \colon \bigcup_{n=1}^\infty I^n \to I$ is called 
\emph{repetition superinvariant} if, for all $n,m\in\N$ and $(x_1,\dots,x_n)\in I^n$, the following inequality is valid
\Eq{*}{
\M(\underbrace{x_1,\dots,x_1}_{m\text{-times}},\dots,\underbrace{x_n,\dots,x_n}_{m\text{-times}})
   \ge\M(x_1,\dots,x_n).
}
Obviously every repetition invariant mean is also repetition superinvariant, however the converse implication is not true in general. Now binding Propositions~\ref{prop:PalPas19aThm2.1} and \ref{prop:Ked} we obtain the next lemma.
\begin{lem}\label{lem:x1xi/x1xn}
 Let $\M$ be a symmetric and Jensen-concave mean on $I$. Then for all $n \in \N$ and $x\in I^n$ we have
 \Eq{*}{
\sum_{i=1}^n \Mm_{j=1}^i \Big( x_j , \tfrac{n!}i\Big) \le n \Mst_{i=1}^n \Big(\frac{x_1+\dots+x_i}{i(x_1+\dots+x_n)}, (n-1)!\Big) \cdot (x_1+\dots+x_n).
 }
\end{lem}

However, in view of Proposition~\ref{prop:PalPas19aThm2.1}, $\M_\#$ is a Jensen-concave mean of $\R_+$. Thus by Lemma~\ref{lem:conc->mon} it is monotone. Therefore we obtain the easy-to-see estimation which we formulate as a corollary. 

\begin{cor}\label{cor:H1}
 Let $\M$ be a symmetric and Jensen-concave mean on $I$. Then for all $n \in \N$ and $x\in I^n$ we have
 \Eq{*}{
&\quad \sum_{i=1}^n \Mm_{j=1}^i \Big( x_j , \tfrac{n!}i\Big)
\le n \Mst_{i=1}^n \Big(\frac{1}i, (n-1)!\Big) \cdot (x_1+\dots+x_n).
 }
\end{cor}

The right hand side of this inequality is close to the one which appears in \eq{E:HnM}. To adjust the left hand side we need to add an additional assumption (superinvariace).

\begin{thm}
 Let $\M$ be a symmetric, Jensen-concave, and repetition superinvariant mean on $I$. Then 
 \Eq{Hseqest}{
\Hc_n(\M) \le n \Mst_{i=1}^n \Big(\frac{1}i, (n-1)!\Big).
 }
   In particular 
\Eq{*}{
\Hc(\M) \le \liminf_{n \to \infty} n \Mst_{i=1}^n \Big(\frac{1}i, (n-1)!\Big).
}
\end{thm}
\begin{proof}

 Indeed, since $\M$ is repetition superinvariant for each sequence $x=(x_1,x_2,\dots,x_n)$ by Corollary~\ref{cor:H1} we get
 \Eq{*}{
\sum_{i=1}^n \M(x_1,\dots,x_i) &\le \sum_{i=1}^n \Mm_{j=1}^i \Big( x_j , \tfrac{n!}i\Big) \\
&\le n \Mst_{i=1}^n \Big(\frac{1}i, (n-1)!\Big) \cdot (x_1+\dots+x_n).
 }
which easily yields \eq{Hseqest}. Using this statement, in view of Proposition~\ref{prop:Hn->Hinfty}, we obtain the second part. 
\end{proof}

It is important to note that this theorem alone does not provide us a method to establish the Hardy constant, only its upper bound. Nevertheless, by virtue of Propositions \ref{prop:lower_estim_Hc} and \ref{prop:LowHar}, establishing a lower bound of $\Hc(\M)$ is way easier.

\section{Application to mixed means}
The idea of mixed means starts with a paper Carlson-Meany-Nelson \cite{CarMeaNel71}. The idea was to take the arithmetic mean of all $k$-element subsequences of $(x_1,\dots,x_n)$ and then calculate the geometric mean of the so-obtained vector of length $\binom nk$. Obviously we 
can replace arithmetic and geometric means be two others. This idea was developed in a quasiarithmetic setting by Sadikova~\cite{Sad06}. Such kind of means are well-defined if and only if $n \ge k$ (possibly fulfilled by $n=1$ since the value of a mean is trivial in this case). It cause some technical problems with regards to the Hardy property (cf.~\cite{Pas15b}). For this reason we restrict our consideration to the case $k=2$ only. We also consider slight modification of this idea which takes into account all pairs of elements, that is a vector of length $n^2$ (see definition below for details). 

It turns out that it does not affect the Hardy constant for a vast family of mixed power means. We conjecture that it is a general property.

For a symmetric mean $\M \colon \bigcup_{n=1}^\infty I^n \to I$ and $\Nm \colon I^2 \to I$ let us set two means $\M \circ \Nm \colon \bigcup_{n=1}^\infty I^n \to I$ and $\M \csqs \Nm \colon \bigcup_{n=1}^\infty I^n \to I$ by 
\Eq{*}{
\M \circ \Nm(x_1,\dots,x_n)&:=
\begin{cases}
x_1 &\quad \text{ for }n=1, \\
\M\big(\Nm(x_i,x_j) \colon 1\le i<j\le n\big) &\quad \text{ for }n\ge 2; \\
                             \end{cases}\\
\M \csqs \Nm(x_1,\dots,x_n)&:=
%\begin{cases}
%x_1 &\quad \text{ for }n=1, \\
\M\big(\Nm(x_i,x_j) \colon i,j \in\ \N_n\big). %&\quad \text{ for }n\ge 2. \\
%\end{cases}
}

Let us start with few easy-to-see properties.
\begin{rem}\label{rem:mix}
 Let $\M \colon \bigcup_{n=1}^\infty I^n \to I$ be a symmetric mean and $\Nm \colon I^2 \to I$. If both $\M$ and $\Nm$ are monotone (resp. homogeneous, concave or convex) then so are $\M \circ \Nm$ and $\M \csqs \Nm$.
 
Furthermore $\M \csqs \Nm$ is symmetric and if, additionally, $\Nm$ is symmetric then so is $\M \circ \Nm$.
\end{rem}

Next properties are just a bit more complicated. 

\begin{lem}\label{lem:compmix}
 Let $\M \colon \bigcup_{n=1}^\infty I^n \to I$ be a symmetric mean, $\Nm \colon I^2 \to I$, and $L,U\colon \bigcup_{n=1}^\infty I^n \to I$ be two monotone, associative means with $L \le \M \le U$ and $L \le \Nm \le U$. Then $L \le \M \circ \Nm \le U$ and $L \le \M \csqs \Nm \le U$.
\end{lem}
Indeed, we have $\M \circ \Nm \le U \circ \Nm \le U \circ U =U$. The remaining inequalities are analogous.

\begin{lem}\label{lem:mix}
 Let $\M \colon \bigcup_{n=1}^\infty I^n \to I$ be a symmetric and repetition (super)invariant mean and $\Nm \colon I^2 \to I$ be a bivariate mean. Then $\M \csqs \Nm$ is a symmetric and repetition (super)invariant mean.
 \end{lem}
\begin{proof}
Applying Remark~\ref{rem:mix} we know that $\M\csqs \Nm$ is symmetric. Now set $n,m \in \N$ and $x \in I^n$. Let $y\in I^{nm}$ be a vector $x$ repeated $m$ times, i.e. 
 \Eq{*}{
 y_{nk+l}:=x_l \qquad \text{ where }k \in\{0,\dots,m-1\}\text{ and }l \in \N_n.
% y:=(\underbrace{x_1,\dots,x_n,\dots,x_1,\dots,x_n}_{\text{vector }(x_1,\dots,x_n)\text{ repeated }m\text{ times}}).
 }
 For $i,j \in \N_{nm}$ denote briefly $\nu_{i,j}:=\Nm(y_i,y_j)$.  Obviously $\nu_{i+n,j}=\nu_{i,j}$ for all $(i,j) \in \N_{(n-1)m} \times \N_{nm}$
 and $\nu_{i,j}=\nu_{i,j+n}$ for all $(i,j) \in \N_{nm} \times \N_{n(m-1)}$. Whence the vector 
 \Eq{*}{
 w:=(\nu_{1,1},\dots,\nu_{1,nm},\nu_{2,1},\dots,\nu_{2,nm},\dots,\nu_{nm,1},\dots,\nu_{nm,nm}) \in I^{(nm)^2}
 }
 contains each element of $(\nu_{i,j})_{i,j\in \N_n}$ exactly $m^2$ times.
 Thus, since $\M$ is symmetric and repetition superinvariant, we get
 \Eq{*}{
\M \csqs \Nm(y)=\M(w)\ge \M \big((\nu_{i,j})_{i,j\in \N_n} \big)=\M \csqs \Nm(y_1,\dots,y_n)=\M \csqs \Nm(x),
}
which shows that $\M \csqs \Nm$ is a superinvariant mean, too. 

To obtain the repetition invariant counterpart of this result we can simply replace the inequality sign above by the equality.
\end{proof}

\subsection{Mixed power means} In the narrow case when both $\M$ and $\Nm$ are power means these operators admit a number of additional properties. Hardy means among this family was studied by the author in the paper \cite{Pas15b}. In particular it was proved (cf. Corollary~1 therein) that 
\begin{enumerate}
 \item $\P_p\circ \P_q$ is a Hardy mean whenever $p<1$ or ($p=1$ and $q \le 0$);
 \item $\P_p\circ \P_q$ is not a Hardy mean whenever $p \ge 2$ or ($p \ge 1$ and $q > 0$).
\end{enumerate}
The Hardy property of $\P_p\circ \P_q$ in the remaining case, that is $(p,q) \in (1,2) \times (-\infty,0]$ remains the open problem.
We deliver some estimations of the Hardy constant for these means and establish its precise value in the case $q<p\le 1$.

\begin{lem}\label{lem:6}
 Let $p,q \in \R$. Then, for all $n \in \N$ with $n \ge 2$ and $v \in \R_+^n$, we get
\Eq{*}{
\P_p \circ \P_q(v)=\begin{cases}
                              \Big (\frac{n}{n-1} (\P_p \csqs \P_q(v))^p - \frac{1}{n-1} (\P_p(v))^p\Big)^{1/p} & \text{ for }p \ne 0;\\
                              \Big ((\P_0 \csqs \P_q(v))^n (\P_0(v))^{-1}\Big)^{1/(n-1)} & \text{ for }p = 0.
                             \end{cases}
}
\end{lem}
Proof of this lemma is just a straightforward calculations. Due to this fact, in view of Lemma~\ref{lem:compmix} we obtain our next result.

\begin{lem}\label{lem:7}
Let $p,q \in \R$ with $p>q$. Then $\P_p \circ \P_q$ is a repetition superinvariant mean. Moreover for all $n \in \N$ and $v \in \R_+^n$ we get 
\Eq{*}{
\P_p \circ \P_q (v_1,\dots,v_n)&\le \lim_{m \to \infty} \P_p \circ \P_q (\underbrace{v_1,\dots,v_n,\dots,v_1,\dots,v_n}_{\text{ vector }v\text{ repeated }m\text{ times}})\\
&= \P_p \csqs \P_q(v_1,\dots,v_n).
}
\end{lem}
Its proof is clear in view of repetition invariance of $\P_p \csqs \P_q$, Lemma~\ref{lem:6} and the equality $\P_p \csqs \P_q \le \P_p$.
In what follows we provide the lower estimation of the Hardy constant of $\P_p \csqs \P_q$ and $\P_p \circ \P_q$.

% \begin{proof}
% Let $v \in \R_+^n$ be an arbitrary vector and $w$ be a vector $v$ repeated $k$ times. Then
% \Eq{*}{
% \M_3(w)&=\M_2(w)^{\frac {kn}{kn-1}}\P_0(w)^{-\frac1{kn-1}}=\M_2(v)^{\frac {kn}{kn-1}}\P_0(v)^{-\frac1{kn-1}}\\
% &=\M_2(v)^{\frac {n}{n-1}}\P_0(v)^{-\frac1{n-1}} \big( \tfrac{\P_0(v)}{\M_2(v)}\big)^{\frac{(k-1)n}{(kn-1)(n-1)}}\\
% &=\M_3(v) \big( \tfrac{\P_0(v)}{\M_2(v)}\big)^{\frac{(k-1)n}{(kn-1)(n-1)}} \ge \M_3(v).
% }
% Therefore $\M_3$ is a repetition superinvariant mean.
%  
% \end{proof}
% 

\begin{lem}\label{lem:estP-P}
For all $p,q \in \R$ we have $\Est(\P_p \csqs \P_q)=\varrho_{p,q}$, where
\Eq{varrho}{
\varrho_{p,q}=
\begin{cases}
 \Big( \iint_{[0,1]^2} \big( \frac{x^{-q}+y^{-q}}{2}\big)^{\frac pq}\:dx\:dy\Big)^{\frac1p} & \text{ if }pq\ne 0;\\
\exp \Big( \frac1q \iint_{[0,1]^2}\ln \big( \frac{x^{-q}+y^{-q}}{2}\big)\:dx\:dy\Big) & \text{ if }p=0,q\ne 0;\\
\Big(\int_0^1 x^{-\frac p2}\:dx\Big)^{\frac2p} & \text{ if }p\ne 0,q= 0;\\
e & \text{ if }p=q= 0. 
\end{cases}
}
In particular this value is finite if and only if the corresponding integral on the right hand side is finite.

Moreover $\Est(\P_p \circ \P_q)=\varrho_{p,q}$ for all $p  \in (-\infty,1)$ and $q\in \R$.
\end{lem}

\begin{proof}
 In the case $pq \ne 0$, since power means are homogeneous we have
 \Eq{*}{
\Est(\P_p \csqs \P_q)&= \lim_{n \to \infty} n  \P_p \csqs \P_q \big(1,\tfrac12,\dots,\tfrac1n\big)
=\lim_{n \to \infty} \P_p \big( \P_q \big(\tfrac ni, \tfrac nj \big) \colon i,j \in \N_n\big)\\
&=\lim_{n \to \infty} \Bigg( \frac{1}{n^2} \sum_{i,j \in \N_n} \Bigg( \frac{\big(\frac ni \big)^q+\big(\frac ni \big)^q}{2}\Bigg)^{\frac pq}\Bigg)^{\frac1p}\\
&=\lim_{n \to \infty} \Bigg( \frac{1}{n^2} \sum_{i,j \in \N_n} \Bigg( \frac{\big(\frac in \big)^{-q}+\big(\frac in \big)^{-q}}{2}\Bigg)^{\frac pq}\Bigg)^{\frac1p}\\
&=\bigg( \iint_{[0,1]^2} \Big( \frac{x^{-q}+y^{-q}}{2}\Big)^{\frac pq}\:dx\:dy\bigg)^{\frac1p}.
}
In the second case when $p=0$ and $q \ne 0$ one gets
 \Eq{*}{
\Est(\P_0 \csqs \P_q)&=\lim_{n \to \infty} n  \P_0 \csqs \P_q \big(1,\tfrac12,\dots,\tfrac1n\big)
=\lim_{n \to \infty} \P_0 \big( \P_q \big(\tfrac ni, \tfrac nj \big) \colon i,j \in \N_n\big)\\
&=\lim_{n \to \infty} \exp \Bigg( \frac{1}{n^2} \sum_{i,j \in \N_n} \frac1q \ln \Bigg( \frac{\big(\frac ni \big)^q+\big(\frac ni \big)^q}{2}\Bigg)\Bigg)\\
&=\lim_{n \to \infty} \exp \Bigg( \frac{1}{q n^2} \sum_{i,j \in \N_n} \ln \Bigg( \frac{\big(\frac in \big)^{-q}+\big(\frac in \big)^{-q}}{2}\Bigg)\Bigg)\\
&=\exp \Bigg( \frac{1}{q } \iint_{[0,1]^2}\ln \bigg( \frac{x^{-q}+y^{-q}}{2}\bigg)\:dx\:dy\Bigg).
}
In the third case $p \ne 0$ and $q=0$ we obtain  
 \Eq{*}{
\Est(\P_p \csqs \P_0)&=\lim_{n \to \infty} n  \P_p \csqs \P_0 \big(1,\tfrac12,\dots,\tfrac1n\big)
=\lim_{n \to \infty} \P_p \big( \P_0 \big(\tfrac ni, \tfrac nj \big) \colon i,j \in \N_n\big)\\
&=\lim_{n \to \infty} \bigg( \tfrac{1}{n^2} \sum_{i,j \in \N_n} \Big(\sqrt{\tfrac ni \cdot \tfrac nj }\Big)^p\bigg)^{\frac1p}\\
&=\lim_{n \to \infty} \bigg( \tfrac1n \sum_{i \in \N_n} \big(\tfrac ni\big)^{\frac p2} \cdot \tfrac1n \sum_{j \in \N_n} \big(\tfrac nj\big)^{\frac p2} \bigg)^{\frac1p}
=\lim_{n \to \infty} \bigg( \tfrac{1}{n} \sum_{i=1}^n \big(\tfrac ni\big)^{\frac p2}\bigg)^{\frac2p}\\
&=\bigg(\lim_{n \to \infty}  \tfrac{1}{n} \sum_{i=1}^n \big(\tfrac in\big)^{-\frac p2}\bigg)^{\frac2p}
=\bigg(\int_0^1 x^{-\frac p2}\:dx\bigg)^{\frac2p}.%=\big(\tfrac{2}{2-p}\big)^{\frac1p}
}
Finally in case $p=q=0$ we get $\P_0 \csqs \P_0=\P_0$. Thus using the well-known result for the geometric mean we obtain
$\Est(\P_0 \csqs \P_0)=\Est(\P_0)=e$, and the proof of the first part is complete.

Now we proceed to the moreover part. First fix $p \in (-\infty,1) \setminus \{0\}$ and $q \in \R$. Under this assumption by Lemma~\ref{lem:6} we obtain 
\Eq{*}{
&\quad \Est(\P_p \circ \P_q)=\lim_{n \to \infty} \P_p \circ \P_q \big(n,\tfrac n2,\dots,\tfrac nn\big)\\
&=\lim_{n \to \infty} 
\Big (\tfrac{n}{n-1} (\P_p \csqs \P_q(n,\tfrac n2,\dots,\tfrac nn))^p - \tfrac{1}{n-1} (\P_p(n,\tfrac n2,\dots,\tfrac nn))^p\Big)^{1/p}
}
But sequences $(\P_p \csqs \P_q(n,\tfrac n2,\dots,\tfrac nn))_{n=1}^\infty$ and $(\P_p(n,\tfrac n2,\dots,\tfrac nn))_{n=1}^\infty$ are convergent to $\Est(\P_p \csqs \P_q)$ and $\Est(\P_p)$, respectively (in particular the second limit is finite). Now we can pass to the limit termwise to obtain the equality $\Est(\P_p \circ \P_q)=\Est(\P_p \csqs \P_q)$.
To prove this equality for $p=0$ we need to use the second alternative in Lemma~\ref{lem:6}.
\end{proof}

\begin{thm}\label{thm:2}
For all $p,q \in (-\infty,1]$ the equality $\Hc(\P_p \csqs \P_q)=\varrho_{p,q}$ holds, where the value of $\varrho_{p,q}$ is defined by \eq{varrho}. 

If, additionally, $p>q$ then we also have $\Hc(\P_p \circ \P_q)=\varrho_{p,q}$.
\end{thm}
\begin{proof}
Let $p,q \in (-\infty,1]$. Then both $\P_p$ and $\P_q$ are monotone, symmetric and concave. Thus, by Remark~\ref{rem:mix}, so is $\P_p \csqs \P_q$. Moreover, by Lemma~\ref{lem:mix}, it is also repetition invariant.
Whence, by Proposition~\ref{prop:PalPas16Thm3.4} and Lemma~\ref{lem:estP-P}, we have $\Hc(\P_p \csqs \P_q)=\Est(\P_p \csqs \P_q)=\varrho_{p,q}$.

Furthermore for all $q<p\le1$ by Lemma~\ref{lem:7} we get $\P_p \circ \P_q \le \P_p \csqs \P_q$.
Now using Proposition~\ref{prop:LowHar} we can reapply Lemma~\ref{lem:estP-P} to obtain
\Eq{*}{
\varrho_{p,q}=\Est(\P_p \circ \P_q) \le \Hc(\P_p \circ \P_q) \le \Hc(\P_p \csqs \P_q)=\varrho_{p,q},
}
which completes our proof.
\end{proof}

We now show a simple application of this result.
\begin{exa}
The following inequalities are valid for all $x \in \ell_1(\R_+)$:
\Eq{*}{
\sum_{n=1}^\infty \sqrt[n^2]{\prod_{i=1}^n\prod_{j=1}^n \frac{x_i+x_j}2} \le 2\sqrt{e} \sum_{n=1}^\infty x_n;\\
x_1+\sum_{n=2}^\infty \sqrt[\binom n2\ ]{\prod_{1 \le i < j \le n} \frac{2x_ix_j}{x_i+x_j}} \le \frac{e^{3/2}}2 \sum_{n=1}^\infty x_n;\\
\sum_{n=1}^\infty \sqrt[n^2]{\prod_{i=1}^n\prod_{j=1}^n \frac{2x_ix_j}{x_i+x_j}} \le \frac{e^{3/2}}2 \sum_{n=1}^\infty x_n.
}
Moreover all constants on the right hand sides are sharp.
\end{exa}
\begin{proof}
Using the notion of mixed means and Hardy property we can rewrite the latter inequalities in a brief form
\Eq{*}{
\Hc(\P_0 \csqs \P_{1})&=2\sqrt{e}\quad \text{ and } \quad
 \Hc(\P_0 \circ \P_{-1})=\Hc(\P_0 \csqs \P_{-1})&=\frac{e^{3/2}}2.
 }
By Theorem~\ref{thm:2} we shall calculate $\varrho_{0,-1}$ and $\varrho_{0,1}$ given by \eq{varrho}. To this end, let us start with the integral part of the first value 
\Eq{*}{
\iint_{[0,1]^2}  \ln \big(\tfrac{x+y}2\big) \:dx\:dy&=
\int_0^2 (1-\abs{1-t}) \ln(t)\:dt-\ln 2\\
&=\int_0^1 t \ln(t)\:dt+\int_1^2 (2-t) \ln(t)\:dt-\ln 2\\
&=-\tfrac14 +\ln 4 - \tfrac 54 - \ln 2 = \ln 2 -\tfrac32.
}
Then, by \eq{varrho}, we have 
\Eq{*}{
\varrho_{0,-1}&=\exp \Big( - \iint_{[0,1]^2}\ln \big( \tfrac{x+y}{2}\big)\:dx\:dy\Big)=\exp(\tfrac32 - \ln 2)=\frac{e^{3/2}}2.
}

In the second step we evaluate the product $\varrho_{0,1} \cdot \varrho_{0,-1}$. However, by virtue of \eq{varrho}, it equals
\Eq{*}{
\varrho_{0,1}\cdot \varrho_{0,-1}&=
\exp \Big(\iint_{[0,1]^2}\ln \big( \frac{x^{-1}+y^{-1}}{2}\big)-\ln \big( \frac{x+y}{2}\big)\:dx\:dy\Big)\\
&=\exp \Big(-\iint_{[0,1]^2}\ln (xy)\:dxdy\Big)=\exp \Big(-2\int_0^1 \ln x\:dx\Big)=e^2.
}
Finally we have 
\Eq{*}{
\varrho_{0,1}=\frac{e^2}{\varrho_{0,-1}}=e^2 \cdot \frac2{e^{3/2}}=2\sqrt{e},
}
which ends our proof. 
\end{proof}

%\bibliography{publ,funcequ,computsci,newbib,preprints,fixed}
%\bibliographystyle{plain}

\end{document}